\documentclass[a4paper]{amsart}
\usepackage[utf8]{inputenc}
\usepackage[T1]{fontenc}
\usepackage{lmodern}
\usepackage{amsmath, amsthm, amssymb, amsfonts}
\usepackage[all]{xy}
\usepackage{nicefrac,mathtools}
\usepackage[shortlabels, inline]{enumitem}
\usepackage{microtype}
\usepackage{hyperref}
\usepackage{graphicx}
\usepackage[all]{xy}
\usepackage{xcolor}
\usepackage{tikz-cd}
\usepackage{thmtools}

\numberwithin{equation}{section}
\theoremstyle{plain}
\newtheorem{theorem}[equation]{Theorem}

\newtheorem{lemma}[equation]{Lemma}
\newtheorem{proposition}[equation]{Proposition}

\theoremstyle{definition}
\newtheorem{definition}[equation]{Definition}

\theoremstyle{remark} 
\newtheorem{remark}[equation]{Remark}

\newcommand*{\Hilm}[1][H]{\mathcal #1}
\newcommand{\Bound}{\mathbb B}
\newcommand*{\Cst}{\mathrm{C}^*}
\newcommand{\norm}[1]{\lvert\!\lvert #1\rvert\!\rvert}
\newcommand{\C}{\mathbb{C}}
\newcommand{\N}{\mathbb{N}}

\newcommand*{\nb}{\nobreakdash}
\newcommand*{\defeq}{\mathrel{\vcentcolon=}}

\title[Partial Crossed Products]{Inductive limits of partial crossed products}

\author{Md Amir Hossain} \email{mdamirhossain18@gmail.com}

\address{The Institute of Mathematical Sciences, A CI of Homi Bhabha National Institute, 4th cross street, CIT Campus, Taramani, Chennai, 600113, India.
}


\begin{document}
	\begin{abstract}
		Let \(\big((A^{(i)}, G, \alpha^{(i)}), \phi_i\big)_{i \in \N}\) be an inductive sequence of partial dynamical systems. 
		We prove the existence of an induced partial action \(\alpha\) of \(G\) on the inductive limit \(A=\varinjlim A^{(i)}\). We call \(\alpha\) the inductive limit partial action. Furthermore, we show the corresponding partial crossed product~\(A\rtimes_{\alpha}G\) is canonically isomorphic to \(\varinjlim A^{(i)}\rtimes_{\alpha^{(i)}}G\). We also study the globalization of the inductive limit partial action \(\alpha\), its finite Rokhlin dimension and tracial states on~\(A\rtimes_{\alpha}G\).  
	\end{abstract}
	
	\maketitle
	
	\section{Introduction}
	
	Partial actions of groups on \(\Cst\)\nb-algebras---also called partial dynamical systems---were introduced independently by Exel~\cite{Exel-1994-Circle-action-Cst-alg} and McClanahan~\cite{McClanahan-1995-K-theory-Par-act}. Given such a system, one can construct a new \(\Cst\)\nb-algebra, the partial crossed product, which encodes essential information about the underlying partial dynamics. In particular, the partial dynamical system and its partial crossed product share the same representation theory~\cite[Theorem 13.2]{Exel2017Book-Partial-action-Fell-bundle}.
	
	Many interesting \(\Cst\)-algebras can be described as partial crossed products (see, for example,~\cite{Exel1994-Bunce-Deddens-alg-as-crossed-product-by-partial-Automorphism, Exel-1995-AF-alg-and-partial-action}). Such descriptions often make it possible to compute their 
	K-theory~\cite{Exel-1994-Circle-action-Cst-alg} and to analyze properties such as the amenability of the groupoid naturally associated with the partial action~\cite{Anantharaman-D:2020Partial-Action-of-Gpd}.

	The inductive limit is a fundamental tool for constructing a new \(\Cst\)-algebra from existing ones. 
	It is particularly useful for computing the K-theory of the inductive limit \(\Cst\)\nb-algebra and for understanding its structural properties.
	This article focuses on studying the inductive limit of partial actions. We define an inductive sequence of partial dynamical systems \(\big((A^{(i)}, G, \alpha^{(i)}), \phi_i\big)_{i\in \N}\) (see~Definition~\ref{def-ind-lim-par-dyn}), where \(G\) is a discrete group. We then establish the existence of a unique partial action \(\alpha\) of \(G\) on \(A= \varinjlim A^{(i)}\). We call \(\alpha\) the \emph{inductive limit partial action}
	and the triple \((A,G,\alpha)\) is called \emph{inductive limit partial dynamical system}. Our main result (Theorem~\ref{thm-ind-par-crosd-pord}) establishes an isomorphism between \(A\rtimes_\alpha G\) and the inductive limit \(\varinjlim A^{(i)}\rtimes_{\alpha^{(i)}}G\).
	
	A natural source of partial actions arises from restricting global actions to (possibly non-invariant) ideals. Such partial actions are called globalizable.
	However, not every partial action of a group on a \(\Cst\)\nb-algebra is globalizable (see~\cite[p.\ 213]{Exel-1994-Circle-action-Cst-alg}), so it is an interesting question to determine whether a given partial action is globalizable. As an application of our main result, we investigate the globalization of inductive limit partial action. 
	
	We have also studied the Rokhlin dimension of the inductive limit partial action and tracial states on the associated partial crossed product.

	\smallskip
	
	\paragraph{\itshape Structure of the article:}
	Section~\ref{sec:Prelim} contains preliminaries about partial dynamical systems and partial crossed products. Section~\ref{sec:ind-lim-par-act} contains the main theorem of the article, which identifies the partial crossed product of the inductive limit partial action with the inductive limit of individual partial crossed products (see Theorem~\ref{thm-ind-par-crosd-pord}). In Section~\ref{sec:application}, we present three applications of our main result: Theorem~\ref{thm:globalize} proves that if every partial action of an inductive sequence is globalizable, then the inductive limit partial action is also globalizable. Propositions~\ref{prop-rok-dim} and~\ref{prop:trace} discuss the finite Rokhlin dimension of the inductive limit partial action and tracial states on the partial crossed product associated to the inductive limit partial action.   
	
	\section{Preliminaries}\label{sec:Prelim}
	
	We follow Exel's book~\cite{Exel2017Book-Partial-action-Fell-bundle} for partial actions. Throughout the article, we fix a discrete group \(G\) and denote its identity element by \(e\). The set of all bounded linear operators on a separable Hilbert space \(\Hilm\) is denoted by \(\Bound(\Hilm)\).
	
	\begin{definition}[{\cite[Definition 11.4]{Exel2017Book-Partial-action-Fell-bundle}}]
		A \emph{partial action} of \(G\) on a \(\Cst\)\nb-algebra \(A\) is a pair
		\(\alpha= (\{A_g\}_{g\in G}, \{\alpha_{g}\}_{g\in G})\), consisting of closed two-sided ideals \(\{A_g\}_{g\in G}\) of \(A\) and \(^*\)\nb-isomorphisms \(\alpha_{g}\colon A_{g^{-1}} \to A_{g}\) for \(g\in G\) satisfying 
		\begin{enumerate}
			\item \(A_e= A\) and \(\alpha_{e} \colon A_e\to A_e\) is the identity map;
			\item\label{def-par-act-2}  \(\alpha_{g}\circ \alpha_{h} \subseteq \alpha_{gh}\) for \(g,h\in G\).
		\end{enumerate} 
	\end{definition}
	The domain of \(\alpha_{g}\circ \alpha_{h}\) is given by \(\{x\in A_{h^{-1}}: \alpha_h(x) \in A_{g^{-1}}\} = \alpha^{-1}_h(A_{g^{-1}})\). Condition~\eqref{def-par-act-2} of the above definition says that \(\alpha_g(\alpha_h(x)) = \alpha_{gh}(x)\) for \(x\in \textup{dom}(\alpha_g\circ \alpha_h)\), that is, \(\alpha_{gh}\) is an extension of \(\alpha_g\circ \alpha_h\).

	We call \((A,G, \alpha)\) a \(\Cst\)\nb-algebraic partial dynamical system. Let \(A\rtimes_{alg}G\) denote the \(^*\)\nb-algebra consisting of all finite formal sums \(\sum_{g\in G} a_g\delta_g\) with \(a_g\in A_g\). The multiplication and involution on \(A\rtimes_{alg}G\) are given by
	\[
	a\delta_g \cdot b\delta_h \defeq a\alpha_{g}(b)\delta_{gh} \quad \textup{and} \quad (a\delta_g)^* \defeq \alpha_{g^{-1}}(a^*)\delta_{g^{-1}}
	\]
	for \(a\in A_g, b\in A_h\) and \(g,h\in G\). If \(p\) is a  \(\Cst\)\nb-seminorm on \(A\rtimes_{alg}G\), then
	\[
	p(a) \leq \sum_{g\in G} \norm{a_g}
	\] 
	for \(a=\sum_{g\in G} a_g\delta_g \in A\rtimes_{alg}G\) (see~\cite[Proposition 11.9]{Exel2017Book-Partial-action-Fell-bundle}). Define a norm on \(A\rtimes_{alg}G\) by 
	\[
	\norm{a}_{\text{max}}\defeq \sup\{p(a) : p \textup{ is a \(\Cst\)\nb-seminorm on } A\rtimes_{alg}G\}. 
	\]
	The completion of \(A\rtimes_{alg}G\) with respect to the norm \(\norm{\cdot}_{\text{max}}\) is called the \emph{partial crossed product} and is denoted by \(A\rtimes_{\alpha}G\).
	
	Consider the partial dynamical system \((A,G,\alpha)\). The map \(i\colon A \to A\rtimes_{\alpha}G\) defined by \(a\mapsto a\delta_e\) is an embedding. Moreover, there exists a conditional expectation \(E\colon A\rtimes_{\alpha}G \to A\) by 
	\begin{equation}\label{equ-cond-exp}
		E(a\delta_g) = \begin{cases}
			a & \textup{ if } g=e,\\
			0 & \textup{ if } g\neq e
		\end{cases}
	\end{equation}
	for \(a\in A_g\) and \(g\in G\).
	
	\begin{definition}\label{def-partial-rep}
		Let \(G\) be a discrete group and let \(A\) be a unital \(\Cst\)\nb-algebra. A map \(u\colon G \to A\) by \(g\mapsto u_g\)  is called a \emph{partial representation} of \(G\) in \(A\) if it satisfies
		\begin{enumerate}
			\item \(u_e = 1\),
			\item  \(u_g u_h u_{h^{-1}} = u_{gh} u_{h^{-1}}\),
			\item \(u_{g^{-1}} u_gu_h = u_{g^{-1}} u_{gh}\),
			\item  \(u_{g^{-1}} = u_{g}^*\)
		\end{enumerate}
		for \(g,h\in G\).
	\end{definition}
	In Definition~\ref{def-partial-rep}, if \(A = \Bound(\Hilm)\) for a separable Hilbert space \(\Hilm\), then we call \(u\) a partial group representation on \(\Hilm\). 
	\begin{definition}\label{def-sub-rep}
		Let \(u^{(i)} \colon G \to \Bound(\Hilm_i)\) be partial group representations for \(i=1,2\). We call \(u^{(1)}\) a \emph{sub-partial representation} of \(u^{(2)}\) if \(\Hilm_1\subseteq \Hilm_2\) and 
		\[
		u^{(2)}_{g}(h) = u^{(1)}_{g}(h)
		\] 
		for \(g\in G\) and \(h\in \Hilm_1\).
	\end{definition}

	\begin{lemma}\label{lem:ind-gp-par-Hilb}
		Suppose for each \(i\in \N\), we are given a partial representation \(u^{(i)} \colon G \to \Bound(\Hilm_i)\), and if \(i\leq j\), then \(u^{(i)}\) is a sub-partial representation of \(u^{(j)}\). Then there exists a unique partial representation \(u\colon G\to \Bound(\Hilm)\) with \(u|_{\Hilm_i} = u^{(i)}\) and \(\Hilm = \overline{\bigcup_{i=1}^{\infty} \Hilm_i}\).  	
	\end{lemma}
	\begin{proof}
		Let \(g\in G\). Define \(u_g \colon \bigcup_{i=1}^{\infty} \Hilm_i \to \bigcup_{i=1}^{\infty} \Hilm_i\) by 
		\[
		u_g(h_i) = u^{(i)}_g(h_i)
		\]
		for \(h_i\in \Hilm_i\). Since \(u^{(i)}\) is a sub-partial representation of \(u^{(j)}\) whenever \(i\leq j\), the map \(u_g\) is well-defined, and \(u_g\) is bounded on \(\bigcup_{i=1}^{\infty} \Hilm_i\). Thus, \(u_g\) can be extended to a bounded linear map on \(\Hilm\). Since each \(u^{(i)}\) is a partial representation, by definition, \(u\) is a partial representation.
		
		If \(v\) is another partial representation of \(G\) in \(\Bound(\Hilm)\) such that \(v|_{\Hilm_i} = u^{(i)}\), then \(v\) agree with \(u\) on the dense subspace \(\bigcup_{i=1}^{\infty}\Hilm_i\) and hence \(u=v\).
	\end{proof}
	A covariant representation of a partial dynamical system \((A, G,\alpha)\) on a Hilbert space \(\Hilm\) is a pair \((\pi, u),\) where \(\pi\colon A \to \Bound(\Hilm)\) is a \(^*\)\nb-homomorphism and \(u\colon G \to \Bound(\Hilm)\) is a partial representation satisfying
	\[
	u_g\pi(a)u_{g^{-1}} = \pi(\alpha_g(a))
	\]
	 for \(a\in A_{g^{-1}}, g\in G\). Such a covariant representation \((\pi, u)\) gives a unique representation, called the \emph{integrated form}, \(\pi \rtimes u \colon A\rtimes_{\alpha}G \to \Bound(\Hilm)\) characterized by \(\pi\rtimes u(a\delta_g) = \pi(a)u_g\) for \(a\in A_{g}, g\in G\). Theorem~13.2 of~\cite{Exel2017Book-Partial-action-Fell-bundle} ensures that given a nondegenerate representation \(\rho\colon A\rtimes_{\alpha}G \to \Bound(\Hilm) \), there exists a unique covariant representation \((\pi,u)\) of \((A, G, \alpha)\) on \(\Hilm\) such that \(\rho = \pi \rtimes u\).  
	
	We conclude this section with the following useful lemma.
	\begin{lemma}[{\cite[Lemma 9.4.26]{Phillips2011-Crossed-prod-book}}]\label{lem:extend-rep-H}
		Let \(\pi\colon A \to \Bound(\Hilm)\) be a representation of a \(\Cst\)\nb-algebra \(A\) on a Hilbert space \(\Hilm\). Let \(\Hilm_0\) be the closed linear span of \(\pi(A)\Hilm\). Then
		\begin{enumerate}
			\item \(\Hilm_0\) is an invariant subspace for \(\pi\);
			\item the representation \(\pi_0 = \pi|_{\Hilm_0}\) is nondegenerate;
			\item \(\Hilm_0^{\perp} = \{h\in \Hilm : \pi(a)h = 0 \textup{ for all } a\in A \}\);
			\item \(\pi\) is the direct sum of \(\pi_0\) and the zero representation on \(\Hilm_0^{\perp}\);
			\item  for all \(a\in A\), we have \(\norm{\pi_0(a)} = \norm{\pi(a)}\);
			\item \(\ker(\pi_0) = \ker(\pi)\);
			\item if \(\pi\) is the direct sum of a nondegenerate representation on a Hilbert space \(\Hilm_1\) and the zero representation on \(\Hilm_2\), then \(\Hilm= \Hilm_1\) and \(\Hilm^{\perp} = \Hilm_2\).
		\end{enumerate}
	\end{lemma}

	\section{Inductive limits of partial crossed products}\label{sec:ind-lim-par-act}

	An inductive sequence in a category \(\mathcal{C}\) is a sequence of objects \((A^{(i)})_{i\in \N}\) in \(\mathcal{C}\) together with morphisms \(\phi_i\colon A^{(i)} \to A^{(i+1)}\) in \(\mathcal{C}\) for each \(i\). We shall denote an inductive sequence by \((A^{(i)}, \phi_i)_{i\in \N}\).
	
	\begin{definition}[{\cite[Definition 6.2.2]{Rordam2002K-theory-book}}]
		An inductive limit of an inductive sequence \((A^{(i)}, \phi_i)_{i \in \N}\) in a category \(\mathcal{C}\) is a pair \((A, (\mu_{i})_{i \in \N})\), where \(A\) is an object in \(\mathcal{C}\) and \(\mu_i\colon A^{(i)} \to A\) is a morphism in \(\mathcal{C}\) such that \(\mu_{i+1}\circ \phi_i = \mu_{i}\) for \(i\in \N\). Furthermore, this limit is unique in the following sense: if there is another limit \((B,(\lambda_i)_{i \in \N})\), then there is one and only one morphism \(\lambda \colon A \to B\) satisfying \(\lambda \circ \mu_i = \lambda_i\) for \(i \in \N\).  
	\end{definition}
	
	Inductive limits may or may not exist in a category.
	For instance, inductive limits do not exist in the category of finite sets. In contrast, inductive limits do exist in the category of \(\Cst\)\nb-algebras (cf.~\cite[Proposition 6.2.4]{Rordam2002K-theory-book}).
	\begin{definition}\label{def-ind-lim-par-dyn}
		We call \(\big((A^{(i)}, G, \alpha^{(i)}), \phi_i\big)_{i\in \N}\) an inductive sequence of partial dynamical systems if 
		\begin{enumerate}
			\item \((A^{(i)}, G, \alpha^{(i)})_{i\in \N}\) is a partial dynamical system for \(i\in \N\);
			\item \(\phi_i\colon A^{(i)} \to A^{(i+1)}\) is a \(G\)\nb-equivariant \(^*\)\nb-homomorphism for \(i\in \N\);
			\item \((A^{(i)}, \phi_i)_{i\in \N}\) is an inductive sequence of \(\Cst\)\nb-algebras.
		\end{enumerate} 
	\end{definition}
	
	\begin{proposition}\label{prop-ind-dyn-sys}
		Let \(G\) be a discrete group. Consider an inductive sequence of partial dynamical systems \(\big((A^{(i)}, G, \alpha^{(i)}), \phi_i\big)_{i\in \N}\) with \(A = \varinjlim A^{(i)}\). Then there exists a unique partial action \(\alpha\) of \(G\) on \(A\) satisfying \(\alpha_{g} = \varinjlim \alpha^{(i)}_{g}\) for \(g\in G\).
	\end{proposition}
	\begin{proof}
		Suppose \((A,(\mu_i)_{i\in \N})\) is the inductive limit of \((A^{(i)}, \phi_i)_{i\in \N}\). Let 
		\[
		\alpha^{(i)} = \big(\{A^{(i)}_g\}_{g\in G}, \{\alpha^{(i)}_{g}\}_{g\in G}\big)
		\]
		for \(i\in \N\). Since \(\phi_i\) is \(G\)\nb-equivariant, we have \(\phi_i(A^{(i)}_g) \subseteq A^{(i+1)}_g\) for \(i\in \N\) and \(g\in G\). Therefore, \(\big(A^{(i)}_g, \phi_i\big)_{i\in \N}\) is an inductive sequence with the inductive limit \((A_g, (\mu_i)_{i\in \N})\) for \(g\in G\) (we are using the same notations \(\phi_i\) and \(\mu_i\) in the place of \(\phi_i|_{A^{(i)}_g}\) and \(\mu_i|_{A^{(i)}_g}\), respectively). From~\cite[Proposition 6.2.4]{Rordam2002K-theory-book}, we have \(A_g = \overline{\bigcup_{i=1}^{\infty}\mu_i(A^{(i)}_g)}\) for \(g\in G\). Since \(A^{(i)}_g\) is a closed two-sided ideal of \(A^{(i)}\) for \(i\in  \N\) and \(g\in G\), the limit \(A_g = \varinjlim A^{(i)}_g\) is a closed two-sided ideal of \(A = \varinjlim A^{(i)}\). The \(G\)\nb-equivariance of \(\phi_i\) gives the following commutative diagram 
		\begin{center}
			\begin{tikzcd}
				A^{(1)}_{g^{-1}} \arrow{r}{\phi_1} \arrow{d}{\alpha^{(1)}_{g}}
				& A^{(2)}_{g^{-1}} \arrow{r}{\phi_2} \arrow{d}{\alpha^{(2)}_{g}} & A^{(3)}_{g^{-1}} \arrow[r, dashed] \arrow{d}{\alpha^{(3)}_{g}}
				& A^{(i)}_{g^{-1}} \arrow{r}{\phi_i} \arrow{d}{\alpha^{(i)}_{g}} & A^{(i+1)}_{g^{-1}} \arrow[r, dashed] \arrow{d}{\alpha^{(i+1)}_{g}}
				& \varinjlim A^{(i)}_{g^{-1}} \arrow{d}{\alpha_{g}}\\
				A^{(1)}_g \arrow{r}{\phi_1} & A^{(2)}_g \arrow{r}{\phi_2} & A^{(3)}_g \arrow[r, dashed] & A^{(i)}_g \arrow{r}{\phi_i} & A^{(i+1)}_g \arrow[r, dashed] & \varinjlim A^{(i)}_g.
			\end{tikzcd}
		\end{center}
		Moreover, since \(\alpha^{(i)}_{g} \colon A^{(i)}_{g^{-1}} \to A^{(i)}_g\) is a \(^*\)\nb-isomorphism for \(i\in \N\) and \(g\in G\),
		the universal property of the inductive limit ensures the existence of a
		 unique \(^*\)\nb-isomorphism 
		 \[
		 \alpha_g\colon A_{g^{-1}} = \varinjlim A^{(i)}_{g^{-1}}\to A_g = \varinjlim A^{(i)}_g
		 \]
		 such that 
		\begin{equation}\label{equ-G-equiv-mu}
			\alpha_g (\mu_i(a)) = \mu_i(\alpha^{(i)}_{g}(a)) \quad \textup{for } a\in A^{(i)}_{g^{-1}}.
		\end{equation}
		
		We now show that \(\alpha = (\{A_g\}_{g\in G}, \{\alpha_{g}\}_{g\in G})\) is a partial action of \(G\) on \(A\). Let \(e\) be the identity element of \(G\). Since \(A^{(i)}_e = A^{(i)}\) and \(\alpha^{(i)}_e = id_{A^{(i)}}\), we have
		\[
		A_e = \varinjlim A^{(i)}_e = \varinjlim A^{(i)} = A \quad \textup{and} \quad \alpha_e = id_{A}.
		\]
	   Let \(g,h\in G\) and \(x\in \textup{dom}(\alpha_g\circ\alpha_h)\). First, we assume \(x = \mu_i(a_{h^{-1}})\) for \(a_{h^{-1}} \in A^{(i)}_{h^{-1}}\) with \(\alpha_h(\mu_i(a_{h^{-1}})) \in A_{g^{-1}} = \overline{\bigcup_{i=1}^{\infty} \mu_i(A^{(i)}_{g^{-1}})}\). 
		But Equation~\eqref{equ-G-equiv-mu} gives us \(\alpha_h(\mu_i(a_{h^{-1}})) = \mu_i(\alpha^{(i)}_h(a_{h^{-1}}))\). This ensures that \(a_{h^{-1}} \in \textup{dom} (\alpha^{(i)}_{g} \circ \alpha^{(i)}_{h})\). Since \(\alpha^{(i)}\) is a partial action, we have \(a_{h^{-1}} \in \textup{dom}(\alpha^{(i)}_{gh})\). Hence, \(x = \mu_i(a_{h^{-1}}) \in \textup{dom}(\alpha_{gh})\). Finally, using Equation~\eqref{equ-G-equiv-mu} we obtain
	\begin{align*}
			\alpha_g\circ \alpha_h(x) &= \alpha_g(\alpha_h(\mu_i(a_{h^{-1}}))) = \alpha_g\big( \mu_i(\alpha^{(i)}_h(a_{h^{-1}}))\big) = \mu_i(\alpha^{(i)}_g \circ \alpha^{(i)}_h(a_{h^{-1}})) \\
		&= \mu_i(\alpha^{(i)}_{gh} (a_{h^{-1}})) = \alpha_{gh}(\mu_i(a_{h^{-1}})) = \alpha_{gh}(x).
	\end{align*}
		Now, if \(x\in \alpha_{h^{-1}}\big(\overline{\bigcup_{i=1}^{\infty} \mu_{i}(A^{(i)}_{g^{-1}})} \big)\), then using the continuity of \(\alpha_g, \alpha_h\) and \(\mu_i\) we conclude that \(x\in \textup{dom}(\alpha_{gh})\) and \(\alpha_g\circ \alpha_h(x) = \alpha_{gh}(x)\).
		Hence, \(\alpha\) is a partial action of~\(G\) on~\(A\).
	\end{proof}
	
	\noindent The partial action \(\alpha\) constructed in Proposition~\ref{prop-ind-dyn-sys} is called \emph{the inductive limit partial action} of \(G\) on the inductive limit \(\Cst\)\nb-algebra \(A = \varinjlim A^{(i)}\). Accordingly, \((A,G, \alpha)\) is referred to as the inductive limit partial dynamical system.
	\begin{remark}
		Equation~\eqref{equ-G-equiv-mu} ensures that the map \(\mu_i \colon A^{(i)} \to A \) is \(G\)\nb-equivariant with respect to the partial dynamical system \((A^{(i)}, G, \alpha^{(i)})\) and the inductive limit partial dynamical system \((A, G,\alpha)\).
	\end{remark}
	
	\begin{theorem}\label{thm-ind-par-crosd-pord}
		Let \(\big((A^{(i)}, G, \alpha^{(i)}), \phi_i  \big)_{i\in \N}\) be an inductive sequence of partial dynamical systems with the inductive limit partial dynamical system \((A, G, \alpha)\) as in Proposition~\ref{prop-ind-dyn-sys}. Then there exists an inductive sequence of partial crossed product \(\big( A^{(i)} \rtimes_{\alpha^{(i)}}G, \psi_i \big)_{i \in \N}\) and 
		\[
		A\rtimes_{\alpha}G \cong \varinjlim A^{(i)} \rtimes_{\alpha^{(i)}} G.
		\]
	\end{theorem}
	\begin{proof}
		Since \(\phi_i\colon A^{(i)} \to A^{(i+1)}\) is a \(G\)\nb-equivariant \(^*\)\nb-homomorphism, there exists a \(^*\)\nb-homomorphism \(\psi_i\colon A^{(i)} \rtimes_{\alpha^{(i)}} G\to A^{(i+1)} \rtimes_{\alpha^{(i+1)}} G\) such that 
		\[
		\psi_i(a\delta_g) = \phi_i(a)\delta_g
		\]
		for \(a\in A^{(i)}, g\in G\) (see~\cite[Proposition 22.2]{Exel2017Book-Partial-action-Fell-bundle}). Moreover, \(\big( A^{(i)} \rtimes_{\alpha^{(i)}} G, \psi_i\big)_{i\in \N}\) is an inductive sequence of \(\Cst\)\nb-algebras as \((A^{(i)}, \phi_i)_{i\in \N}\) is an inductive sequence of \(\Cst\)\nb-algebras. We now show that \(A\rtimes_{\alpha}G\) satisfies the universal property for the inductive limit \(\varinjlim A^{(i)} \rtimes_{\alpha^{(i)}} G\). Let \((A, (\mu_i)_{i\in \N})\) be the inductive limit of the inductive sequence \((A^{(i)}, \phi_i)_{i \in \N}\). Define a \(^*\)\nb-homomorphism \(\lambda_i\colon A^{(i)} \rtimes_{\alpha^{(i)}}G \to A\rtimes_{\alpha}G\) by 
		\[
		\lambda_i(a\delta_g) = \mu_i(a) \delta_g
		\]
		where \(a\in A^{(i)}_g, g\in G\) and \(i\in \N\). For \(a\in A^{(i)}_g\), we have 
		\[
		\lambda_{i+1} \circ \psi_i (a\delta_g) = \lambda_{i+1}(\phi_i(a) \delta_g) = \mu_{i+1}(\phi_i(a)) \delta_g
		= \mu_i(a) \delta_g = \lambda_i(a\delta_g).
		\]
		The third equality above follows from the fact that \((A, (\mu_i)_{i\in \N})\) is the inductive limit of \((A^{(i)}, \phi_i)_{i \in \N}\). Therefore, \(\lambda_{i+1}\circ \psi_i = \lambda_i\) for all \(i \in \N\).
		
		We now prove the uniqueness part. Let \(B\) be a \(\Cst\)\nb-algebra and \(\zeta_i\colon A^{(i)} \rtimes_{\alpha^{(i)}} G \to B\) a \(^*\)\nb-homomorphism such that \(\zeta_{i+1}\circ\psi_i = \zeta_i\) for all \(i\in \N\). We need to show that there is a unique \(^*\)\nb-homomorphism \(\zeta\colon A\rtimes_\alpha G\to B\) such that \(\zeta \circ \lambda_i = \zeta_i\) for \(i\in \N\). 
		
		Without loss of generality, assume that \(B\) is a nondegenerate \(\Cst\)\nb-subalgebra of~\(\Bound(\Hilm)\) for a separable Hilbert space \(\Hilm\). For \(i\in \N\), set \(\Hilm_i = \overline{\zeta_i(A^{(i)}\rtimes_{\alpha^{(i)}} G) \Hilm}\), which is a closed subspace of \(\Hilm\). By Theorem~13.2 of~\cite{Exel2017Book-Partial-action-Fell-bundle}, there exists a unique covariant representation \((\pi_i, u^{(i)})\) of \(\alpha^{(i)}\) on \(\Hilm_i\) such that \(\pi_i \rtimes u^{(i)} = \zeta_i(\cdot)|_{\Hilm_i}\). Using Lemma~\ref{lem:extend-rep-H}, we extend \(\pi_i\) to a representation on \(\Hilm\) by declaring it to be the zero representation on \(\Hilm^{\perp}_i\). By construction of \(\Hilm_i\), we have \(\Hilm_i \subseteq \Hilm_j\) whenever \(i\leq j\) and \(\Hilm_i\) is an invariant subspace for \(u^{(j)}_g\) for all \(g\in G\). Hence, \(u^{(i)}\) is a sub-partial representation of \(u^{(j)}\) for \(i\leq j\) (see Definition~\ref{def-sub-rep}). Moreover, \(\pi_{i+1} \circ \phi_i|_{\Hilm_i} = \pi_i\), and both the representations \(\pi_{i+1} \circ \phi_i\) and \(\pi_i\) vanish on \(\Hilm_{i+1}\cap \Hilm^{\perp}_i\) and \(\Hilm^{\perp}_{i+1}\). Lemma~\ref{lem:extend-rep-H} gives us \(\pi_{i+1} \circ \phi_i = \pi_i\) for all \(i\in \N\). Finally, since \(B\) is a nondegenerate subalgebra of \(\Bound(\Hilm)\), we have \(\Hilm = \overline{\bigcup_{i=1}^{\infty}\Hilm_i}\). 
		
		Thus from the universal property of the inductive limit \(\varinjlim A^{(i)}\), we obtain a representation \(\pi \colon A \to \Bound(\Hilm)\) satisfying \(\pi\circ \phi_i = \pi_i\) for \(i\in \N\).
		
		Lemma~\ref{lem:ind-gp-par-Hilb} ensures that there is a unique partial group representation \(u\colon G \to \Bound(\Hilm)\) such that \(u|_{\Hilm_i} = u^{(i)}\) for \(i\in \N\). Furthermore, by construction and the uniqueness, the pair \((\pi, u)\) is a covariant representation of \((A, G,\alpha)\). Therefore,~\cite[Proposition 13.1]{Exel2017Book-Partial-action-Fell-bundle} gives us a unique \(^*\)\nb-homomorphism \(\zeta\colon A\rtimes_{\alpha}G \to \Bound(\Hilm)\) such that 
		\[
		\zeta \circ \lambda_i = \zeta_i
		\]  
		for \(i\in \N\). The representation \(\zeta\) is the integrated form \(\pi \rtimes u\) of the covariant representation \((\pi,u)\). Since \(A = \overline{\bigcup_{i=1}^{\infty}\mu_i(A^{(i)})}\), we have \(\zeta(A\rtimes_{\alpha}G) \subseteq B\). The uniqueness of the covariant representation \((\pi, u)\) ensures that the integrated form~\(\zeta\) is unique.  
	\end{proof}

	\section{Applications}\label{sec:application}
	
	\subsection{Globalization of the inductive action}
	
	In this section, we discuss globalization, that is, whether a partial action is a restriction of a global action. We prove that if all the partial actions in an inductive sequence have globalization, then the inductive limit partial action also has globalization.
	\begin{definition}[{\cite[Definition 28.1]{Exel2017Book-Partial-action-Fell-bundle}}]\label{def:global}
		Let \(\eta\) be a (global) action of \(G\) on a \(\Cst\)\nb-algebra \(B\) and \(A\) a closed two-sided ideal of \(B\). Suppose \(\beta\) is the partial action obtained by restricting \(\eta\) to the ideal \(A\). We say that \(\eta\) is a globalization of~\(\beta\) if
		\[
		B= \overline{\sum_{g\in G} \eta_g(A)}.
		\] 
	\end{definition}
	Globalization of a given partial action on a \(\Cst\)\nb-algebra may not always exist. However, once globalization exists, it is unique (see~\cite[Proposition 28.2]{Exel2017Book-Partial-action-Fell-bundle}).   
	
	\begin{proposition}\label{prop-globa-equiv-homo}
		Let \(\alpha^{(i)} = \big( \{A^{(i)}_g\}_{g\in G}, \{\alpha^{(i)}_g\}_{g\in G}\} \big)\) be a partial action of \(G\) on \(A^{(i)}\) for \(i=1,2\). Suppose that \(\eta^{(i)}\) is the globalization of \(\alpha^{(i)}\) on \(B^{(i)}\) for \(i=1,2\). If \(\phi\colon A^{(1)} \to A^{(2)}\) is a \(G\)\nb-equivariant \(^*\)\nb-homomorphism, then there exists an equivariant \(^*\)\nb-homomorphism \(\psi\colon B^{(1)} \to B^{(2)}\) extending \(\phi\), that is, \(\psi|_{A^{(1)}} = \phi\).
	\end{proposition}
	\begin{proof}
		Since \(\eta^{(i)}\) is the globalization of \(\alpha^{(i)}\), we have 
		\[
		B^{(i)} = \overline{\sum_{g\in G} \eta^{(i)}_g(A^{(i)})}
		\]
		for \(i=1,2\). Define a linear map \(\psi\colon B^{(1)} \to B^{(2)}\) by 
		\[
		\psi\biggr( \sum_{i=1}^{n} \eta^{(1)}_{g_i} (a_i)\biggr) = \sum_{i=1}^{n}\eta^{(2)}_{g_i}(\phi(a_i))
		\]
		for \(g_i\in G, a_i\in A^{(1)}\) and \(i=1,2,\cdots, n\).
		The map \(\psi\) is bounded as 
		\[
		 \norm{\psi\bigr( \sum_{i=1}^{n} \eta^{(1)}_{g_i} (a_i)\bigr)} = \norm{ \sum_{i=1}^{n}\eta^{(2)}_{g_i}(\phi(a_i))} =\norm{ \sum_{i=1}^{n}\phi (\eta^{(1)}_{g_i}(a_i))} \leq \norm{\sum_{i=1}^{n} \eta^{(1)}_{g_i} (a_i)}.
		\]
		For \(a\in A^{(1)}\), we have 
		\[
		\psi(a) = \psi(\eta^{(1)}_e(a)) = \eta^{(2)}_e(\phi(a)) = \phi(a)
		\]
		where \(e\) is the identity element of \(G\). Therefore, \(\psi|_{A^{(1)}} = \phi\). Moreover, we have 
		\begin{align*}
			\psi\bigr(\eta^{(1)}_g\bigr(\sum_{i=1}^{n}\eta^{(1)}_{g_i}(a_i)\bigr)\bigr)
			&= \psi \bigr(\sum_{i=1}^{n} \eta^{(1)}_{gg_i}(a_i)\bigr) =\sum_{i=1}^{n} \eta^{(2)}_{gg_i}(\phi(a_i)) \\
			&= \eta^{(2)}_g\bigr( \sum_{i=1}^{n}\eta^{(2)}_{g_i}(\phi(a_i)) \bigr) =\eta^{(2)}_g\bigr(\psi\bigr(\sum_{i=1}^{n}\eta^{(1)}_{g_i}(a_i)\bigr)\bigr)
		\end{align*}
		for \(\sum_{i=1}^{n}\eta^{(1)}_{g_i}(a_i) \in B^{(1)}\) and \(g\in G\). Therefore, \(\psi\) is \(G\)\nb-equivariant. 
		To prove \(\psi\) is multiplicative, it is enough to verify the following equality 
		\[
		\psi\big(\eta^{(1)}_g(a)\eta^{(1)}_h(b) \big) = \eta^{(2)}_g(\phi(a))\eta^{(2)}_h(\phi(b))
		\]
		for \(a,b\in A^{(1)}\) and \(g,h\in G\).
		The above equation follows from the following computation
		\begin{align*}
				\psi\big(\eta^{(1)}_g(a)\eta^{(1)}_h(b) \big) &= \psi \big( \eta^{(1)}_h(\eta^{(1)}_{h^{-1}g}(a)b)\big)= \eta^{(2)}_h\big(\phi(\eta^{(1)}_{h^{-1}g}(a)b)\big)\\ &
				= \eta^{(2)}_h\big(\phi(\eta^{(1)}_{h^{-1}g}(a)) \phi(b)\big)
			=  \eta^{(2)}_h\big(\eta^{(2)}_{h^{-1}g}(\phi(a)) \phi(b)\big)\\
			&= \eta^{(2)}_g(\phi(a))\eta^{(2)}_h(\phi(b)).
		\end{align*}
		The second equality of the second line above follows from \(G\)\nb-equivariance of \(\phi\).
		Finally,
		\begin{align*}
				\psi\bigr(\big(\eta^{(1)}_g(a)\big)^*\bigr) &= \psi\big(\eta^{(1)}_{g^{-1}} (a^*)\big) = \eta^{(2)}_{g^{-1}}\big(\phi(a^*)\big) = \eta^{(2)}_{g^{-1}}\big(\phi(a)^*\big) \\
			&= \eta^{(2)}_g(\phi(a))^* = \psi(\eta^{(1)}_g(a))^*
		\end{align*}
		for \(a\in A^{(i)}, g\in G\).
		Therefore, \(\psi\) is a \(^*\)\nb-homomorphism from \(B^{(1)} \) to \(B^{(2)}\).
	\end{proof}
	
	\begin{theorem}\label{thm:globalize}
		Let \(\big((A^{(i)}, G, \alpha^{(i)}), \phi_i  \big)_{i\in \N}\) be an inductive sequence of partial dynamical systems with the inductive limit partial dynamical system \((A, G, \alpha)\). If each partial action \(\alpha^{(i)}\) in the sequence admits a globalization, then the inductive limit partial action \(\alpha\) also admits a globalization.
	\end{theorem}
	\begin{proof}
		Let \(B^{(i)}\) be a \(\Cst\)\nb-algebra and \(\eta^{(i)}\) be the globalization of the partial action \(\alpha^{(i)}\) on \(A^{(i)}\) for \(i\in \N\). Then from Definition~\ref{def:global}, we have 
		\[
		B^{(i)} = \overline{\sum_{g\in G}\eta^{(i)}_g(A^{(i)})}
		\]
		for \(i\in \N\). Since \(\phi_i\colon A^{(i)} \to A^{(i+1)}\) is a \(G\)\nb-equivariant \(^*\)\nb-homomorphism, Proposition~\ref{prop-globa-equiv-homo} gives a \(G\)\nb-equivariant \(^*\)\nb-homomorphism \(\psi_i\colon B^{(i)} \to B^{(i+1)}\) with \(\psi_i|_{A^{(i)}} = \phi_i\) for \(i\in \N\). Moreover, \((B^{(i)}, \psi_i)_{i\in \N}\) is an inductive sequence of \(\Cst\)\nb-algebras. Let \((B,(\mu_i)_{i\in \N})\) be the inductive limit of \((B^{(i)}, \psi_i)_{i\in \N}\). Proposition~8.2.24 of~\cite{Phillips2011-Crossed-prod-book} ensures that there exists a unique inductive (global) action \(\eta = \varinjlim \eta^{(i)}\) of \(G\) on \(B=\varinjlim B^{(i)}\). We claim that \(\eta\) is the globalization of \(\alpha = \varinjlim \alpha^{(i)}\).
		
		We first assume \(a \in \bigcup_{i=1}^{\infty}\mu_{i}(A^{(i)}_{g^{-1}})\). Then
		\(a = \mu_{i}(a^{(i)}) \) for some \(a^{(i)} \in A^{(i)}_{g^{-1}}\). Now
		\[
		\alpha_{g}(a) = \alpha_{g}(\mu_i(a^{(i)})) = \mu_i(\alpha^{(i)}_{g}(a^{(i)})) = \mu_i(\eta^{(i)}_{g}(a^{(i)})) = \eta_{g}(\mu_i(a^{(i)})) = \eta_g(a).
		\]
		If \(a\in A_{g^{-1}} = \overline{\bigcup_{i=1}^{\infty}\mu_{i}(A^{(i)}_{g^{-1}})}\), then a continuity argument ensures that \(\alpha_g(a) = \eta_{g}(a)\). 
		Therefore, the partial action \(\alpha\) is obtained by the restriction of the global action \(\eta\) on \(A\).
		To complete the proof, we need to show that
		\[
		B = \overline{\sum_{g\in G}\eta_{g}(A)}.
		\]
		Clearly, \(\overline{\sum_{g\in G}\eta_{g}(A)} \subseteq B\). For the reverse inclusion, we assume \(x\in \bigcup_{i=1}^{\infty}\mu_i(B^{(i)})\). Then \(x= \mu_i(b)\) for some \(b\in B^{(i)}\). As \(b\in B^{(i)}\), we can write \(b  = \sum_{i=1}^{k}\eta_{g_i}^{(i)}(a^{(i)})\) for \(a^{(i)} \in A^{(i)}\) and \(k\in \N\). Then, we have
		\[
		x= \mu_{i}(b) = \mu_i\biggr(\sum_{i=1}^{k}\eta_{g_i}^{(i)}(a^{(i)})\biggr) = \sum_{i=1}^{k} \mu_i\big( \eta_{g_i}^{(i)}(a^{(i)}) \big) = \sum_{i=1}^{k} \eta_{g_i}\big( \mu_i(a^{(i)})\big).
		\]
		Hence \( x\in\overline{\sum_{g\in G}\eta_{g}(A)}\).
		If \(x\in B = \overline{\bigcup_{i=1}^{\infty}\mu_i(B^{(i)})}\), then a continuity argument ensures that \(x\in \overline{\sum_{g\in G}\eta_{g}(A)}\).
		Thus, \(\eta\) is the globalization of \(
		\alpha\).
	\end{proof}

	\subsection{Rokhlin dimension of the inductive limit partial action}
	
	Rokhlin-type properties of a dynamical system are very crucial to describe the structure of the associated crossed product algebra (see~\cite{Gardella-Hirshberg-Santiago2021Rokhlin-dim-Crossed-prod}). The Rokhlin dimension for partial action was introduced in~\cite{Abadie-Gradella2022-Partial-Rokh-dim}. In Proposition~\ref{prop-rok-dim}, we shall study the finiteness of the Rokhlin dimension of the inductive limit partial action. 
	
	\begin{definition}[{\cite[Definition 2.1]{Abadie-Gradella2022-Partial-Rokh-dim}}] \label{def-Rok-dim}
		Let \(\alpha = (\{A_g\}_{g\in G}, \{\alpha_{g}\}_{g\in G})\) be a partial action of a finite group \(G\) on a unital \(\Cst\)\nb-algebra \(A\). Let \(k\in \N\). We say \(\alpha\) has Rokhlin dimension at most \(k\), denoted \(\textup{dim}_{\textup{Rok}}(\alpha) \leq k\), if for every \(\epsilon >0\) and every finite set \(F\subseteq A\) there exist positive contractions \(f_{j,g}\in A_g\) for \(g\in G\) and \(j=0,1,\cdots,k\) satisfying:
		\begin{enumerate}
			\item \label{Rok-con-1} \(\norm{\alpha_{g}(f_{j,h}x) - f_{j,gh}\alpha_{g}(x)} <\epsilon\) for \(g,h\in G, j = 0,1\cdots,k\) and \(x\in A_{g^{-1}}\cap F\);
			\item \label{Rok-con-2} \(\norm{f_{j,g}f_{j,h}} < \epsilon \) for \(j=0,1,\cdots,k\) and \(g,h\in G\) with \(g\neq h\);
			\item \label{Rok-con-3}\(\norm{\sum_{j=0}^{k} \sum_{g\in G} f_{j,g} -1} < \epsilon\);
			\item \label{Rok-con-4}\(\norm{f_{j,g}a-af_{j,g}} < \epsilon\) for \(j=0,1,\cdots,k\) and \(g \in G,a\in F\).
		\end{enumerate}  
		Furthermore, we say the partial action \(\alpha\) has Rokhlin dimension with commuting towers at most \(k\), denoted \(\textup{dim}^{\textup{c}}_{\textup{Rok}}(\alpha) \leq k\), if for every \(\epsilon > 0\) and every finite set \(F\subseteq A\), there exist positive contractions \(f_{j,g}\in A_g\) for \(g\in G\) and \(j=0,1,\cdots,k\) satisfying Conditions~\eqref{Rok-con-1}-\eqref{Rok-con-4} above and 
		\begin{enumerate}
			\setcounter{enumi}{4}
			\item \label{Rok-con-5} \(\norm{f_{j,g}f_{l,h} - f_{l,h}f_{j,g}} < \epsilon\) for \(j,l = 0,1,\cdots,k\) and \(g,h
			\in G\).
		\end{enumerate} 
	\end{definition}
\noindent The Rokhlin dimension of the partial action \(\alpha\) is defined by 
	\[
	\textup{dim}_{\textup{Rok}}(\alpha) =\min\{k\in \N: \textup{dim}_{\textup{Rok}}(\alpha)  \leq k\}.
	\]
	Similarly, the Rokhlin dimension with commuting towers \(\textup{dim}^{\textup{c}}_{\textup{Rok}(\alpha)}\) is also defined.
	\begin{remark}
		Since we are only considering partial actions of groups on \emph{unital} \(\Cst\)\nb-algebras, we can remove the multiplicative witness \(a\in F\) appeared in Conditions~(1)-(5) of~\cite[Definition 2.1]{Abadie-Gradella2022-Partial-Rokh-dim}. (See the remark after Definition~2.1 of~\cite{Abadie-Gradella2022-Partial-Rokh-dim}.) 
	\end{remark}
	
	\begin{proposition}\label{prop-rok-dim}
		Let \(G\) be a finite group. Let \(\big((A^{(i)}, G, \alpha^{(i)}), \phi_i  \big)_{i\in \N}\) be an inductive sequence of partial dynamical systems with unital \(\Cst\)\nb-algebras \(A^{(i)}\) and unital connecting maps \(\phi_i\) for \(i\in \N\). Suppose that \(A = \varinjlim A^{(i)} \) and \(\alpha = \varinjlim \alpha^{(i)}\).
		Then
		\[
		\textup{dim}_{\textup{Rok}}(\alpha) \leq \liminf_{i\to\infty}  \textup{dim}_{\textup{Rok}}(\alpha^{(i)})
		\]
		and
		\[
		\textup{dim}^{\textup{c}}_{\textup{Rok}}(\alpha) \leq \liminf_{i\to\infty}  \textup{dim}^{\textup{c}}_{\textup{Rok}}(\alpha^{(i)}).
		\] 
	\end{proposition}
	\begin{proof}
		The result holds if \(\liminf_{i\to\infty}  \textup{dim}_{\textup{Rok}}(\alpha^{(i)})
		= \infty\). Therefore, we assume that \(\liminf_{i\to\infty}  \textup{dim}_{\textup{Rok}}(\alpha^{(i)})
		< \infty\). Then there exists \(k\in \N\) such that for a given \(m\in \N\) there is \(n\in \N\) with \(n\geq m\) and satisfying \(\textup{dim}_{\textup{Rok}}(\alpha^{(n)}) \leq k\). So we can pass to a subsequence and assume that \(\textup{dim}_{\textup{Rok}}(\alpha^{(n)}) \leq k\) for all \(n\in \N\). Let~\(\epsilon >0\) and let \(F=\{a_1,a_2,\cdots, a_p\}\) be a finite subset of \(A\). Since \(A= \overline{\bigcup_{i=1}^{\infty}\mu_i(A^{(i)})}\), there exists a positive integer \(n\in \N\) and elements \(b_1,b_2,\cdots, b_p\in A^{(n)}\) such that \(\norm{a_j - \mu_n(b_j)} < \frac{\epsilon}{3}\) for \(j=1,2,\cdots,p\). But \(\textup{dim}_{\textup{Rok}}(\alpha^{(n)}) \leq k\). Then for \(\frac{\epsilon}{3}>0\) and the finite subset \(\{b_1,b_2,\cdots,b_p\}\) of \(A^{(n)}\) there are positive contractions \(f^{(n)}_{j,g}\in A^{(n)}_g\) for \(g\in G\) and \(j=0,1,\cdots,k\) satisfying Conditions~\eqref{Rok-con-1}-\eqref{Rok-con-4} of Definition~\ref{def-Rok-dim}. Consider the canonical map \(\mu_n\colon A^{(n)}_g \to A_g\). Clearly, \(\mu_n(f^{(n)}_{j,g})\in A_g\) for \(j=0,1,\cdots,k\) and \(g\in G\) and each \(\mu_n(f^{(n)}_{j,g})\) is a positive contraction. 
		We claim that: for \(j=0,1,\cdots,k\) and \(g\in G\) the positive contractions \(\mu_n(f^{(n)}_{j,g})\in A_g\) satisfy Conditions~\eqref{Rok-con-1}-\eqref{Rok-con-4} of Definition~\ref{def-Rok-dim} for~\(\epsilon >0\) and the finite set \(F\).
		
		\noindent Condition~\eqref{Rok-con-1}: given \(a\in A_{g^{-1}}\cap F\) choose \(b\in A^{(n)}_{g^{-1}}\cap \{b_1,b_2,\cdots, b_p\}\) such that \(\norm{a-\mu_n(b)} <\frac{\epsilon}{3}\). Now, we have
		\begin{align*}
			&\norm{\alpha_g(\mu_n(f_{j,h})a) - \mu_n(f_{j,gh})\alpha_g(a)} \\
			&\leq \norm{\alpha_g(\mu_n(f_{j,h})\mu_n(b)) - \mu_n(f_{j,gh})\alpha_g(\mu_n(b))}\\
			&+ \norm{\alpha_{g}(\mu_n(f_{j,h})(a-\mu_n(b))) - \mu_n(f_{j,gh})\alpha_{g}(a-\mu_n(b))} \\
			&< \norm{\mu_n(\alpha^{(n)}_{g}(f_{j,h}b) - f_{j,gh}\alpha^{(n)}_{g}(b))} +\frac{\epsilon}{3} +\frac{\epsilon}{3} <\epsilon.
		\end{align*}
		\noindent Conditions~\eqref{Rok-con-2} and~\eqref{Rok-con-3} are straightforward verifications.
		
		\noindent Condition~\eqref{Rok-con-4}: given 
		\(a\in F\) choose \(b\in \{b_1,b_2,\cdots, b_p\}\) such that \(\norm{a-\mu_n(b)} <\frac{\epsilon}{3}\). Then we have
		\begin{align*}
			&\norm{\mu_n(f_{j,g})a-a\mu_n(f_{j,g})} \\ 
			&\leq \norm{\mu_n(f_{j,g}) \mu_n(b)-\mu_n(b)\mu_n(f_{j,g})} \\
			&+\norm{\mu_n(f_{j,g})(a-\mu_n(b))-(a-\mu_n(b))\mu_n(f_{j,g})}\\
			&< \norm{\mu_n(f_{j,g}b-bf_{j,g})} +\frac{\epsilon}{3} + \frac{\epsilon}{3} < \epsilon.
		\end{align*}
		The proof of the commuting tower case is similar and is omitted.  
	\end{proof}

	\subsection{Traces on the inductive limit of partial crossed products}
	
	Consider the partial dynamical system \((A,G,\alpha)\). A linear functional \(\tau\colon A \to \C\) is called \(G\)\nb-invariant if 
	\(\tau(a) = \tau(\alpha_{g^{-1}}(a))\) for \(g\in G\) and \(a\in A_g\).

	\begin{lemma}[{\cite[Lemma 3.3]{Scarparo2017-Supramenable-group-partial-act}}]\label{lem:G-inv-trace}
		Let \((A,G,\alpha)\) be a partial dynamical system and let~\(E\) be the conditional expectation from \(A\rtimes_{\alpha}G\) onto \(A\) as Equation~\eqref{equ-cond-exp}. Let~\(\tau\) be a tracial state on~\(A\). Then \(\tau\circ E\) is a tracial state on \(A\rtimes_{\alpha}G\) iff \(\tau\) is \(G\)\nb-invariant.  
	\end{lemma}
	
	\begin{definition}
		Let \(\big((A^{(i)}, G, \alpha^{(i)}), \phi_i  \big)_{i\in \N}\) be an inductive sequence of partial dynamical systems. Let \(\tau_i\) be a given tracial state on \(A^{(i)}\) for \(i\in \N\). Then \((\tau_i)_{i\in \N}\) is called inductive sequence of tracial states if \(\tau_{i+1}\circ \phi_i = \tau_i\) for \(i\in \N\).
	\end{definition}
	
	\begin{lemma}\label{lem:trace-ind-lim}
		Suppose \((\tau_i)_{i\in \N}\) is a 
		\(G\)\nb-invariant inductive sequence of tracial states on an inductive sequence of partial dynamical system \(\big((A^{(i)}, G, \alpha^{(i)}), \phi_i  \big)_{i\in \N}\). Then there exists a \(G\)\nb-invariant tracial state \(\tau = \varinjlim \tau_i\) on \(A =\varinjlim A^{(i)}\).
	\end{lemma}
	\begin{proof}
		Define a linear functional \(\tau\colon \bigcup_{i=1}^{\infty}\mu_i(A^{(i)}) \to \C\) by \(\tau(\mu_i(a_i)) = \tau_i(a_i)\) for \(a_i\in A^{(i)}\). The functional \(\tau\) is well-defined because \(\tau_{i+1}\circ \phi_i = \tau_i\) for \(i\in\N\). Since each \(\tau_i\) is a tracial state, \(\tau\) is a continuous linear functional, and we can extend \(\tau\) to be a tracial state on \(A\). Given \(a_i\in A^{(i)}_g\), we have 
		\[
		\tau(\mu_i(a_i)) = \tau_i(a_i) = \tau_i(\alpha^{(i)}_{g^{-1}} (a_i)) = \tau(\mu_i(\alpha^{(i)}_{g^{-1}} (a_i))) = \tau(\alpha_{g^{-1}} (\mu_i(a_i))).
		\]
		The second equality of the above line follows from the \(G\)\nb-invariance of \(\tau_i\). After passing to the limit, we can conclude that \(\tau\) is a \(G\)\nb-invariant tracial state on \(A\).
	\end{proof}
	
	Let \(\tau = \varinjlim \tau_i\) be the \(G\)\nb-invariant tracial state on \(A = \varinjlim A^{(i)}\) as in the last lemma. Then by Lemma~\ref{lem:G-inv-trace}, \((\varinjlim \tau_i) \circ E\) is a tracial state on \(A\rtimes_{\alpha}G\) where \(E\) is the conditional expectation given by Equation~\eqref{equ-cond-exp}. 
	
	On the other side \(\tau_i\circ E_i\) is a tracial state on \(A^{(i)} \rtimes_{\alpha^{(i)}} G\), where \(E_i\colon A^{(i)} \rtimes_{\alpha^{(i)}} G \to A^{(i)}\) is the conditional expectation for \(i\in\N\). Moreover, \((\tau_i\circ E_i)_{i\in \N}\) is an inductive sequence of tracial states, as \((\tau_i)_{i\in \N}\) was an inductive sequence of tracial states. Therefore, Lemma~\ref{lem:trace-ind-lim} gives us a tracial state \(\varinjlim (\tau_i\circ E_i)\) on \(A\rtimes_{\alpha}G\).
	
	Let \(a\delta_g \in A\rtimes_{\alpha}G\) with \(a\in \bigcup_{i=1}^{\infty}\mu_{i}(A^{(i)})\). Then \(a= \mu_i(a_i)\) for \(a_i\in A^{(i)}\). Now 
	\begin{align*}
		(\varinjlim \tau_i) \circ E (a\delta_g) &=  \begin{cases}
			\varinjlim \tau_i (a) & \textup{ if } g=e,\\
			0 & \textup{ if } g\neq e;
		\end{cases}\\
		& = \begin{cases}
			\tau_i (a_i) & \textup{ if } g=e,\\
			0 & \textup{ if } g\neq e;
		\end{cases}\\
		&=\tau_i\circ E_i(a_i\delta_g)\\
		& =(\varinjlim (\tau_i \circ E_i)) (\mu_i(a_i\delta_g)).
	\end{align*}
	If \(a\in \overline{\bigcup_{i=1}^{\infty}\mu_{i}(A^{(i)})} = A\), then after passing to the limit, we have
	\[
	(\varinjlim \tau_i) \circ E (a\delta_g) = \varinjlim (\tau_i \circ E_i) (a\delta_g).
	\]
	Therefore, we have proved the following proposition. 
	\begin{proposition}\label{prop:trace}
		Let \(\big((A^{(i)}, G, \alpha^{(i)}), \phi_i  \big)_{i\in \N}\) be an inductive sequence of partial dynamical systems with the inductive limit partial dynamical system \((A, G,\alpha)\). 
		Suppose \((\tau_i)_{i\in \N}\) is a \(G\)\nb-invariant inductive sequence of tracial states on \((A^{(i)}, \phi_i)_{i\in \N}\). Then \((\varinjlim \tau_i) \circ E= \varinjlim (\tau_i\circ E_i)\).
	\end{proposition}
	
	\medskip
	
	\paragraph{\itshape Acknowledgement:}
	We thank the anonymous referee for their valuable comments and suggestions, which improved the clarity and readability of the manuscript. We also thank the Institute of Mathematical Sciences (IMSc) for providing financial support.

\end{document}